\documentclass[reqno,12pt,draft]{amsart}
\usepackage{amsmath,amsthm,amsfonts,amssymb}

\date{}
\newtheorem{theorem}{Theorem}[section]
\newtheorem{lemma}[theorem]{Lemma}
\newtheorem{corollary}[theorem]{Corollary}
\newtheorem{remark}[theorem]{Remark}

\newtheorem{example}[theorem]{Example}
\numberwithin{equation}{section}

\begin{document}

\centerline{\bf On estimates of solutions of Fokker--Planck--Kolmogorov equations with }
\centerline{\bf potential terms and non uniformly elliptic diffusion matrices}

\vspace*{0.2cm}

\centerline{\bf S.V. Shaposhnikov}

\vspace*{0.2cm}

\centerline{Department of Mechanics and Mathematics, Moscow State University,}
\centerline{119991 Moscow, Russia, starticle@mail.ru}

\vspace*{0.2cm}

\centerline{{\bf Abstract}}

\vspace*{0.2cm}

{\small
We consider Fokker--Planck--Kolmogorov equations with unbounded coefficients and obtain
upper estimates of solutions. We also obtain
new estimates involving Lyapunov functions.}


\section{\sc Introduction}

The goal of this work is to obtain upper estimates of solutions of the Fokker--Planck--Kolmogorov equation
\begin{equation}\label{e1}
\partial_t\mu=\partial_{x_i}\partial_{x_j}\bigl(a^{ij}\mu\bigr)-\partial_{x_i}\bigl(b^i\mu\bigr)+c\mu.
\end{equation}
Throughout the summation over repeated  indices is meant.
Let $T>0$. We shall say that a locally finite Borel measure $\mu$ on
$\mathbb{R}^d\times(0, T)$ is given by a flow of Borel measures~$(\mu_t)_{t\in (0, T)}$
if for every Borel set $B\subset \mathbb{R}^d$ the mapping $t\to \mu_t(B)$ is measurable
and for every function
$u\in C^{\infty}_0(\mathbb{R}^d\times(0, T))$ one has
$$
\int_{\mathbb{R}^d\times(0, T)}u(x, t)\,\mu(dx\,dt)=\int_0^T\int_{\mathbb{R}^d}u(x, t)\,\mu_t(dx)\,dt.
$$
A typical example is $\mu(B)=P(x_t\in B)\,dt$, where $x_t$ is a random process.
Set
$$
Lu=a^{ij}\partial_{x_i}\partial_{x_j}u+b^i\partial_{x_i}u+cu.
$$
We shall say that a measure $\mu=(\mu_t)_{t\in(0, T)}$ satisfies equation (\ref{e1})
if $a^{ij}$, $b^i$ and $c$ are locally integrable with respect to
the measure $|\mu|$ (the total variation of $\mu$) and
$$
\int_0^T\int_{\mathbb{R}^d}
\bigl[\partial_tu(x, t)+Lu(x, t)\bigr]\,\mu_t(dx)\,dt=0
$$
for every $u\in C^{\infty}_0(\mathbb{R}^d\times(0, T))$.
The measure $\mu$ satisfies the initial condition
$\mu|_{t=0}=\nu$, where $\nu$ is a Borel locally finite measure on $\mathbb{R}^d$, if
for every function $\zeta\in C^{\infty}_0(\mathbb{R}^d)$ there holds the equality
$$
\lim_{t\to 0}\int_{\mathbb{R}^d}\zeta(x)\,\mu_t(dx)=\int_{\mathbb{R}^d}\zeta(x)\,\nu(dx).
$$

The following assertion is trivial and can be found in \cite{BDR08}, \cite{BRSH11}.

\begin{lemma}\label{lem01}
Let $\mu=(\mu_t)_{t\in(0, T)}$ be a solution of equation {\rm (\ref{e1})},
let $u\in C^{1, 2}(\mathbb{R}^d\times(0, T))$ be such that
$u(t, x)=0$ if $x\not\in U$ for some ball $U\subset\mathbb{R}^d$. Then
there exists a set $J_u\subset(0, T)$ of full Lebesgue measure in $(0, T)$ such that
for all $s, t\in J_u$
$$
\int_{\mathbb{R}^d}u(x, t)\,\mu_t(dx)=\int_{\mathbb{R}^d}u(x, s)\,\mu_s(dx)+
\int_s^t\int_{\mathbb{R}^d}
\bigl[\partial_{\tau}u(x, \tau)+Lu(x, \tau)\bigr]\,\mu_{\tau}(dx)\,d\tau.
$$
Moreover, if, in addition, $u\in C(\mathbb{R}^d\times[0, T))$,
the measure $\mu=(\mu_t)_{0<t<T}$ satisfies the initial condition
$\mu|_{t=0}=\nu$ and $a^{ij}$, $b^i, c\in L^1(U\times[0, T], \mu)$,
then  we may assume that for every $t\in J_u$
$$
\int_{\mathbb{R}^d}u(x, t)\,\mu_t(dx)=\int_{\mathbb{R}^d}u(x, 0)\,\nu(dx)+
\int_0^t\int_{\mathbb{R}^d}
\bigl[\partial_{\tau}u(x, \tau)+Lu(x, \tau)\bigr]\,\mu_{\tau}(dx)\,d\tau.
$$
\end{lemma}

We shall say that a Borel measure $\sigma$ is a subprobability on $\mathbb{R}^d$
if $\sigma\ge 0$ and $\sigma(\mathbb{R}^d)\le 1$.
A subprobability measure $\sigma$ on $\mathbb{R}^d$ is probability if $\sigma(\mathbb{R}^d)=1$.

A function $V\in C^{1, 2}(\mathbb{R}^d\times(0, T))\bigcap C(\mathbb{R}^d\times[0, T))$
is termed a Lyapunov function if for every closed interval $[a, b]\subset(0, J)$ one has
$$
\lim_{|x|\to+\infty}\min_{t\in [a, b]}V(x, t)=+\infty.
$$

We shall obtain $L^p$ and $L^{\infty}$ local and global estimates
of the densities of solutions of equation~(\ref{e1}).
Our main interest is in the case of unbounded coefficients of the operator~$L$.
If the coefficients are globally bounded or have a linear growth, then there
are well-known Gaussian estimates (see, e.g., \cite{ArS} and \cite{Ed}).

Global boundedness of the densities (with upper estimates) for solutions
of the Cauchy problem for equation~(\ref{e1}) without any restrictions on the
growth of coefficients is established in \cite{BRSH1} for sufficiently
regular initial conditions. More precisely, the existence of a density of
the initial condition with finite entropy is required.
In \cite{ALR}, \cite{MPR2}, \cite{MPR1} and \cite{Spino} the
transition kernels of the semigroup $\{T_t\}$ are investigated such
that for every nonnegative bounded continuous function $f$
the function $T_tf$ is the minimal nonnegative solution of the Cauchy problem
$\partial_tu=Lu$, $u|_{t=0}=f$.
It is assumed there that the coefficients are locally H\"older continuous and
the diffusion matrix $A$ is uniformly nondegenerate and continuously differentiable.
Moreover, the coefficients do not depend on $t$.
The principal results of the cited papers give certain upper estimates of the kernel densities
and the continuity of semigroup $T_t$ in various functional classes.
The conditions on the coefficients in these papers are  formulated in terms of certain Lyapunov functions.
The kernel of $\{T_t\}$ satisfies equation (\ref{e1}), but the initial condition is Dirac's measure,
so the results from \cite{BRSH1} do not apply.

In \cite{SH0011} and \cite{SH10}, some estimates of densities are
 obtained for arbitrary initial conditions.
The main idea of these works is to deduce global bounds from local estimates in \cite{BKR1}
by using appropriate scalings.
Note that in \cite{SH0011} and \cite{SH10} the coefficients $b$ and $c$ are assumed to be
only integrable, but the diffusion matrix is assumed to uniformly bounded,
uniformly nondegenerate and uniformly Lipschitzian.

In the present work we generalize the results from \cite{SH0011} and \cite{SH10} to the case
where the diffusion matrix can be unbounded and need not be uniformly elliptic.
Moreover, we generalize the estimates  from \cite{MPR1} and \cite{Spino} involving Lyapunov functions.
The main difference between the estimates from  \cite{MPR1}, \cite{Spino}
and the usual estimates with Lypunov functions
is that the former do not depend on the initial condition.

Let us consider an example, illustrating our results.
Assume that $d=d_1+d_2\ge 2$, and will write $x=(x', x'')$,
where $x'\in\mathbb{R}^{d_1}$, $x''\in\mathbb{R}^{d_2}$.
Let $r>2$, $k>r$ and $\delta\in (0, 1)$. Set
$$
A(x, t)=e^{|x'|^{r-\delta}-|x''|^{r-\delta}}I, \quad b(x, t)=-x|x|^{r-2}e^{|x'|^{r-\delta}-|x''|^{r-\delta}}, \quad c(x, t)=-|x|^k.
$$
Let $\varrho$ is a density of a nonnegative solution $\mu=(\mu_t)$, satisfying the condition
$$
\mu_t(\mathbb{R}^d)\le \nu(\mathbb{R}^d)+\int_0^t\int_{\mathbb{R}^d}c(x, s)\,d\mu_s\,ds.
$$
Then there holds the following estimate
$$
\varrho(x, t)\le c_1\exp\bigl(-c_2|x|^r\bigr)\exp\bigl(c_3t^{-\frac{r}{k-r}}\bigr)
$$
for all $(x, t)\in\mathbb{R}^d\times(0, T)$ and some
positive numbers $c_1$, $c_2$ and $c_3$. Note, that the estimate does not depend on 
the initial condition $\nu$.
We consider more general situation in Example (\ref{ex55}).

It is worth mentioning that various lower estimates are considered in \cite{BRSH2}.
The existence and uniqueness problems are investigated
in \cite{BDR08} and \cite{BRSH11}.
A recent survey on elliptic and parabolic equations for measures is given in \cite{BKRU}.

The next section is concerned with estimates involving Lyapunov functions.
In the last section we obtain local and global $L^p$ and $L^{\infty}$ estimates and investigate the
behavior of densities at infinity.

\section{\sc Estimates with Lyapunov functions}

In this section we assume that $c\le 0$ and investigate a solution $\mu$ that is given by
a family of nonnegative measures $\mu_t$ such that $|c|\in L^1(\mu)$ and
\begin{equation}\label{con1}
\mu_t(\mathbb{R}^d)\le \nu(\mathbb{R}^d)+\int_0^t\int_{\mathbb{R}^d}c(x, s)\,\mu_s(dx)\,ds.
\end{equation}
In particular, $\mu_t$ are subprobability measures on $\mathbb{R}^d$.
There are no  other restrictions on the coefficients $a^{ij}$, $b^i$ and $c$.

Note that the kernels considered in \cite{MPR2} satisfy condition (\ref{con1}).
Moreover, in the case of globally bounded coefficients
any solution $\mu$ which is given by a family of  subprobability measures~$(\mu_t)_{t\in(0, T)}$
satisfies condition (\ref{con1}). Note also that if $c$ is continuous and
$\mu_t$ is a weak limit of a sequence of measures $\mu^n_t$ satisfying (\ref{con1}), then
 condition (\ref{con1}) is fulfilled for each~$\mu_t$. Hence this condition is fulfilled
 for every solution $\mu$ obtained as a limit of solutions of equations with bounded coefficients.
Thus, this is a natural condition that is a generalization of the hypothesis
that $\mu_t$ is a subprobability measure for almost all $t$ in the case $c=0$.

\begin{theorem}\label{th1}
Let $\mu=(\mu_t)_{0<t<T}$ be a solution of the Cauchy problem
$\partial_{t}\mu=L^{*}\mu$, \mbox{$\mu|_{t=0}=\nu$} such that
$c\le 0$, $\mu_t$ and $\nu$ are subprobability measures on $\mathbb{R}^d$
and condition~{\rm (\ref{con1})} holds.
Assume that there exists a Lyapunov function $V$ such that
for some positive functions $K, H\in L^1((0, T))$ one has
$$
\partial_tV(x, t)+LV(x, t)\le K(t)+H(t)V(x, t).
$$
Assume also that $V(\, \cdot \,,0)\in L^1(\nu)$.
Then for almost all $t\in(0, T)$
$$
\mu_t(\mathbb{R}^d)=\nu(\mathbb{R}^d)+\int_0^t\int_{\mathbb{R}^d}c(x, s)\,\mu_s(dx)\,ds
$$
and
$$
\int_{\mathbb{R}^d}V(x, t)\,\mu_t(dx)\le Q(t)+R(t)\int_{\mathbb{R}^d}V(x, 0)\,\nu(dx),
$$
where
$$
R(t)=\exp\Bigl(\int_0^tH(s)\,ds\Bigr), \quad Q(t)=R(t)\int_0^t\frac{K(s)}{R(s)}\,ds.
$$
\end{theorem}
\begin{proof}
Let $\zeta_N\in C^2([0, +\infty))$ be such that $0\le \zeta'\le 1$, $\zeta''\le 0$,
and $\zeta_N(s)=s$ if $s\le N-1$ and $\zeta(s)=N$ if $s>N+1$.
Substitute the function $u=\zeta_{N}(V)-N$
in the equality in Lemma~\ref{lem01}. We obtain
\begin{align*}
\int_{\mathbb{R}^d}\zeta_{N}(V(x, t))\,\mu_t(dx)
&=\int_{\mathbb{R}^d}\zeta_{N}(V(x, s))\,\mu_s(dx)+
\\
&+
\left(\mu_t(\mathbb{R}^d)-\nu(\mathbb{R}^d)-\int_0^t\int_{\mathbb{R}^d}c(x, \tau)\,\mu_{\tau}(dx)\,d\tau\right)N+
\\
&+\int_s^t\int_{\mathbb{R}^d}\Bigl(\zeta_N'(V)(\partial_tV
+LV)+\zeta_N''(V)|\sqrt{A}\nabla V|^2\Bigr)\,\mu_{\tau}(dx)\,d\tau+
\\
&+\int_s^t\int_{\mathbb{R}^d}c\left(\zeta_N(V)-\zeta'_N(V)V\right)\,\mu_{\tau}(dx)\,d\tau.
\end{align*}
Noting that $z\zeta'_N(z)\le \zeta_N(z)$, we have
\begin{align*}
\int_{\mathbb{R}^d}\zeta_N(V(x, t))\,\mu_t(dx)
&\le
\int_{\mathbb{R}^d}\zeta_N(V(x, s))\,\mu_s(dx)+
\\
&
\left(\mu_t(\mathbb{R}^d)-\nu(\mathbb{R}^d)-\int_0^t\int_{\mathbb{R}^d}c(x, \tau)\,\mu_{\tau}(dx)\,d\tau\right)N+
\\
&+\int_s^t K(\tau)+H(\tau)\int_{\mathbb{R}^d}\zeta_N(V(x, \tau))\,\mu_{\tau}(dx)\,d\tau,\notag
\end{align*}
Letting $s\to 0$, we arrive at the inequality
\begin{multline}\label{in1}
\int_{\mathbb{R}^d}\zeta_N(V(x, t))\,d\mu_t\le
\\
\le
\int_{\mathbb{R}^d}\zeta_N(V(x, 0))\,d\nu+
\left(\mu_t(\mathbb{R}^d)-\nu(\mathbb{R}^d)-\int_0^t\int_{\mathbb{R}^d}c(x, \tau)\,\mu_{\tau}(dx)\,d\tau\right)N+
\\
+\int_0^tK(\tau)+H(\tau)\int_{\mathbb{R}^d}\zeta_N(V(x, \tau))\,\mu_{\tau}(dx)\,d\tau.
\end{multline}
Since
$$
\mu_t(\mathbb{R}^d)\le \nu(\mathbb{R}^d)+\int_0^t\int_{\mathbb{R}^d}c(x, \tau)\,\mu_{\tau}(dx)\,d\tau,
$$
the last inequality can be rewritten as
\begin{multline*}
\int_{\mathbb{R}^d}\zeta_N(V(x, t))\,\mu_t(dx)\le \int_{\mathbb{R}^d}\zeta_N(V(x, 0))\,\nu(dx)+
\\
+\int_0^tK(\tau)+H(\tau)\int_{\mathbb{R}^d}\zeta_N(V(x, \tau))\,\mu_{\tau}(dx)\,d\tau.
\end{multline*}
Applying Gronwall's inequality we obtain
$$
\int_{\mathbb{R}^d}\zeta_N(V(x, t))\,\mu_t(dx)\le Q(t)+R(t)\int_{\mathbb{R}^d}\zeta_N(V(x, 0))\,\nu(dx).
$$
Letting $N\to\infty$, we obtain the required estimate.
Note that if
$$\mu_t(\mathbb{R}^d)<\nu(\mathbb{R}^d)+\displaystyle\int_0^t\int_{\mathbb{R}^d}c(x, \tau)\,\mu_{\tau}(dx)\,d\tau,$$
then
$$
\int_{\mathbb{R}^d}V(x, t)\,\mu_t(dx)-\int_{\mathbb{R}^d}V(x, 0)\,\nu(dx)-
\int_0^tK(\tau)+H(\tau)\int_{\mathbb{R}^d}V(x, \tau)\,\mu_{\tau}(dx)\,d\tau=-\infty,
$$
which is impossible. Hence
$$
\mu_t(\mathbb{R}^d)=\nu(\mathbb{R}^d)+\int_0^t\int_{\mathbb{R}^d}c(x, \tau)\,\mu_{\tau}(dx)\,d\tau,
$$
which completes the proof.
\end{proof}

\begin{corollary}\label{cor1}
Let  $\mu=(\mu_t)_{0<t<T}$ be a solution
of the Cauchy problem $\partial_{t}\mu=L^{*}\mu$, $\mu|_{t=0}=\nu$,
where $c\le 0$, $\mu_t$ and $\nu$ are subprobability measures on $\mathbb{R}^d$
and condition~{\rm (\ref{con1})} holds.
Let a positive function $W\in C^2(\mathbb{R}^d)$ be such that
$\lim\limits_{|x|\to+\infty}W(x)=+\infty$.

{\rm (i)} If for some number $C>0$ and almost every $(x, t)\in \mathbb{R}^d\times(0, T)$
there holds the inequality
$$
LW(x, t)\le C+CW(x),
$$
then for almost every $t\in(0, T)$ we have
$$
\int_{\mathbb{R}^d}W(x)\,\mu_t(dx)\le \exp(Ct)+\exp(Ct)\int_{\mathbb{R}^d}W(x)\,\nu(dx).
$$

{\rm (ii)} Let $G$ be a positive continuous increasing function on $[0, +\infty)$ such that
$$
\int_1^{+\infty}\frac{ds}{sG(s)}<+\infty.
$$
Let $\eta$ be a continuous function on $[0, T)$ defined by the equality
$$
t=\int_0^{\eta(t)}\frac{ds}{sG(s^{-{\delta}})}, \quad \delta\in (0, 1).
$$
If for some number $C>0$ and almost every $(x, t)\in \mathbb{R}^d\times(0, T)$
there holds the inequality
$$
LW(x, t)\le C-W(x)G(W(x)),
$$
then for almost every $t\in(0, T)$ we have
$$
\int_{\mathbb{R}^d}W(x)\,\mu_t(dx)\le\frac{1}{(1-\delta)\eta^{\delta}(t)}+\frac{C}{\eta(t)}\int_0^t\eta(s)\,ds.
$$

{\rm (iii)} Let $G$ and $\eta$ be the functions mentioned in {\rm (ii)}.
Assume that for some number $C>0$ and almost every $(x, t)\in \mathbb{R}^d\times(0, T)$
there holds the inequality
$$
LW(x, t)+\eta(t)|\sqrt{A(x, t)}\nabla W(x)|^2\le C-W(x)G(W(x)).
$$
Then for almost every $t\in(0, T)$
$$
\int_{\mathbb{R}^d}\exp\bigl(\eta(t)W(x)\bigr)\,\mu_t(dx)\le
\exp\Bigl((1-\delta)^{-1}\eta^{1-\delta}(t)+C\int_0^t\eta(s)\,ds\Bigr).
$$
\end{corollary}
\begin{proof}
In order to prove (i) it is enough to apply
 Theorem \ref{th1} with $H(t)=K(t)=C$ and $V(x, t)=W(x)$.

Let us prove (ii). Let $V(x, t)=\eta(t)W(x)$. Set
$$
\partial_tV(x, t)+LV(x, t)\le \eta'(t)W(x)-\eta(t)W(x)G(W(x))+C\eta(t).
$$
Note that for all nonnegative numbers $\alpha$ and $\beta$
$$
\alpha\beta\le \alpha G^{-1}(\alpha)+\beta G(\beta),
$$
where $G^{-1}$ is the inverse function to $G$. Applying this inequality with
$\alpha=\eta'/\eta$ and $\beta=W$, we obtain
$$
\partial_tV(x, t)+LV(x, t)\le
\eta'(t)G^{-1}\bigl(\eta'(t)/\eta(t)\bigr)+C\eta(t)
=\frac{\eta'(t)}{\eta^{\delta}(t)}+C\eta(t),
$$
because our assumptions imply that
$
\eta'(t)=\eta(t)G(\eta^{-\delta}(t)).
$

Applying Theorem \ref{th1} with $H(t)=0$ and
$K(t)=\frac{\eta'(t)}{\eta^{\delta}(t)}+C\eta(t)$,
we arrive at to the required inequality.

Let us prove (iii). Let $V(x, t)=\exp\bigl(\eta(t)W(x)\bigr)$. Then
$$
\partial_tV(x, t)+LV(x, t)\le
\Bigl[\eta'(t)W(x)-\eta(t)W(x)G(W(x))+C\eta(t)\Bigr]\exp\bigr(\eta(t)W(x)\bigl).
$$
Hence
$$
\partial_tV(x, t)+LV(x, t)\le \Bigl[\frac{\eta'(t)}{\eta^{\delta}(t)}+C\eta(t)
\Bigr]\exp\bigr(\eta(t)W(x)\bigl).
$$
Applying Theorem \ref{th1} with $K(t)=0$ and
$$
H(t)=\frac{\eta'(t)}{\eta^{\delta}(t)}+C\eta(t),
$$
we obtain the required assertion.
\end{proof}

Let us consider several examples.

\begin{example}\label{exm0}{\rm
Set $V(x, t)=|x|^{r}$, where $r\ge 2$. Then
$$
LV(x, t)=r|x|^{r-2}\hbox{\rm trace}\, A(x, t)
+r(r-2)|x|^{r-4}(A(x, t)x, x)
+r|x|^{r-2}(b(x, t), x)+|x|^rc(x, t).
$$
Assume that for some numbers $C_1>0$, $C_2>0$ and all $(x, t)\in \mathbb{R}^d\times[0, T]$ we have
$$
r\hbox{\rm trace}\, A(x, t)
+r(r-2)|x|^{-2}(A(x, t)x, x)
+r(b(x, t), x)+|x|^{2}c(x, t)\le C_1+C_2|x|^2.
$$
Let $|x|^{r}\in L^1(\nu)$. Then
$$
\int_{\mathbb{R}^d}|x|^{r}\,\mu_t(dx)\le e^{C_3t}
+e^{C_3t}\int_{\mathbb{R}^d}|x|^{r}\,\nu(dx)
$$
for almost every $t\in(0, T)$ and some $C_3>0$.}
\end{example}

\begin{example}\label{exm1}{\rm
Set $V(x, t)=\exp(\alpha|x|^{r})$, where $r\ge 2$. Then
\begin{multline*}
LV(x, t)=\exp(\alpha|x|^r)
\Bigl[\alpha r|x|^{r-2}\hbox{\rm trace}\, A(x, t)+
\\
+\alpha r(r-2)|x|^{r-4}(A(x, t)x, x)+\alpha^2r^2|x|^{2r-4}(A(x, t)x, x)+
\\
+\alpha r|x|^{r-2}(b(x, t), x)+c(x, t)\Bigr].
\end{multline*}
Suppose that there exists a number $C_1$ such that for every
$(x, t)\in \mathbb{R}^d\times[0, T]$ we have
\begin{multline*}
\alpha r|x|^{r-2}\hbox{\rm trace}\, A(x, t)+
\\
+\alpha r(r-2)|x|^{r-4}(A(x, t)x, x)+\alpha^2r^2|x|^{2r-4}(A(x, t)x, x)+
\\
+\alpha r|x|^{r-2}(b(x, t), x)+c(x, t)\le C_1.
\end{multline*}
If $\exp(|x|^{r})\in L^1(\nu)$, then
$$
\int_{\mathbb{R}^d}\exp(\alpha|x|^{r})\,\mu_t(dx)\le e^{C_2t}
+e^{C_2t}\int_{\mathbb{R}^d}\exp(\alpha|x|^{r})\,\nu(dx).
$$
for almost all $t\in (0, T)$.
}
\end{example}

\begin{example}\label{exm2}{\rm
Let $k>2$ and $r\ge 2$. Assume that
$$
r\hbox{\rm trace}\, A(x, t)
+r(r-2)|x|^{-2}(A(x, t)x, x)
+r(b(x, t), x)+|x|^2c(x, t)\le C_1-C_2|x|^k,
$$
where $C_1>0$ and $C_2>0$. Then
$$
L|x|^r\le C_3-C_3|x|^{r+k-2}
$$
for some $C_3>0$.
Set $W(x)=|x|^{r}$ and $G(z)=C_3z^{\sigma}$, where $\sigma=(k-2)/r>0$. Hence
$$
LW(x, t)\le C_3-WG(W(x)).
$$
Then $\eta(t)=C_4t^{\frac{1}{\delta\sigma}}$, where $C_4$ depends on
$C_3$, $\delta$ and $\sigma$.
By Corollary \ref{cor1} we obtain the estimate
$$
\int_{\mathbb{R}^d}|x|^r\,\mu_t(dx)\le \frac{\gamma}{t^{\frac{r}{k-2}}},
$$
where $\gamma$ depends on $C_1$, $C_2$, $\delta$, $\sigma$.}
\end{example}

\begin{example}\label{3exm03}{\rm
Let $r>2$ and $k>r$. Assume that
\begin{multline*}
\alpha r|x|^{r-2}\hbox{\rm trace}\, A(x, t)+
\\
+\alpha r(r-2)|x|^{r-4}(A(x, t)x, x)+\alpha^2r^2|x|^{2r-4}(A(x, t)x, x)+
\\
+\alpha r|x|^{r-2}(b(x, t), x)+c(x, t)\le C_1-C_2|x|^k,
\end{multline*}
where $C_1>0$ and $C_2>0$. Then
$$
L\exp(\alpha|x|^r)\le C_3-C_3|x|^{k}\exp(\alpha|x|^r)
$$
for some $C_3>0$.
Set $W(x)=\exp(\alpha|x|^{r})$ and $G(z)=C_3|\ln z|^{\sigma}$ if $z\ge 2$, where $\sigma=\frac{k}{r}>1$.
We obtain
$$
LW(x, t)\le C_3-WG(W(x)).
$$
Then $\eta(t)=C_4\exp(-C_5t^{\frac{-1}{\sigma-1}})$, where $C_4>0$ and $C_5>0$
depend on $C_3$, $\delta$ and $\sigma$.
By Corollary \ref{cor1} we have
$$
\int_{\mathbb{R}^d}\exp(\alpha|x|^r)\,\mu_t(dx)\le \gamma_1\exp\bigl(\frac{\gamma_2}{t^{\frac{r}{k-r}}}\bigr),
$$
where $\gamma_1$ and $\gamma_2$ depend on $C_1$, $C_2$, $\delta$ and $\sigma$.}
\end{example}

\begin{example}\label{3exm04}{\rm
Let $r>2$, $k>2$ and $\alpha>0$. Assume that
\begin{multline*}
\alpha r\hbox{\rm trace}\, A(x, t)
+\alpha r(r-2)|x|^{-2}(A(x, t)x, x)+
\\
+\alpha r(b(x, t), x)+\alpha|x|^2c(x, t)+\alpha^2r^2|x|^{r-2}(A(x, t)x, x)\le C_1-C_2|x|^k,
\end{multline*}
where $C_1>0$ and $C_2>0$. Then
$$
\alpha L|x|^r+\alpha^2r^2|x|^{2r-4}(A(x, t)x, x)\le C_3-C_3|x|^{k+r-2}.
$$
Set $W(x)=\alpha|x|^{r}$ and
$G(z)=C_3\alpha^{-(1+\sigma)/\sigma}z^{\sigma}$, where $\sigma=\frac{k-2}{r}>0$.
We obtain
$$
LW(x, t)+|\sqrt{A(x, t)}\nabla W(x)|^2\le C_3-WG(W(x)).
$$
Hence we can apply Corollary \ref{cor1} with $\delta\in (0, 1)$,
$\eta(t)=C_4t^{\frac{1}{\delta\sigma}}$,
where $C_4$ depends on $C_3$, $\delta$ and $\sigma$.
Thus, for every $\beta>\frac{r}{k-2}$
we obtain the estimate
$$
\int_{\mathbb{R}^d}\exp(\alpha t^{\beta}|x|^r)\,\mu_t(dx)\le
\gamma_1\exp\Bigl(\gamma_2( t^{\beta-\frac{r}{k-2}}+t^{\beta+1})\Bigr),
$$
where the numbers $\gamma_1$ and $\gamma_2$ depend on $C_1$, $C_2$, $r$ and $\beta$.}
\end{example}

Note that the estimates in Examples \ref{3exm03} and \ref{3exm04} do not depend on
the initial condition. If we apply these estimates to the transition probabilities
$P(y, 0, t, \,dx)$ of the corresponding processes, then the resulting estimates will be
uniform  in $y$.
Such estimates for kernels of diffusion semigroups (with possibly rapidly growing drifts)
were first obtained in \cite{MPR1} and \cite{Spino}.

\section{\sc Local and global bounds of solutions}

In this section we obtain local and global $L^p$ and $L^{\infty}$ estimates of densities of solutions.
The main idea is to use a modification of Moser's iteration method (see \cite{Ms}).
We start with local estimates and then
we obtain global estimates by using local one and a suitable scaling.

Let $\mu=(\mu_t)_{t\in (0, T)}$ be a nonnegative solution of equation (\ref{e1}).

We assume that $A=(a^{ij})$ is a symmetric matrix satisfying the following condition:

(H1) for some number $p>d+2$, every ball $U\subset\mathbb{R}^d$ and
every segment $J\subset(0, T)$ one has
$$
\sup_{t\in J}\|a^{i, j}(\, \cdot \,, t)\|_{W^{1, p}(U)}<\infty
$$
and
$$
0<\lambda(U, J):=\inf\bigl\{(A(x, t)\xi, \xi): \, |\xi|=1, \, (x, t)\in U\times J\bigr\}.
$$

We also suppose that

(H2) for some number $p>d+2$, every ball $U\subset\mathbb{R}^d$ and
every closed interval  $J\subset(0, T)$ one has
$$
b, c\in L^p(U\times J) \quad {\rm or} \quad b, c\in L^p(U\times J, \mu).
$$

According to \cite[Corollary 3.9]{BKR1} and \cite[Corollary 2.2]{BRSH3}, conditions (H1) and (H2)
yield existence of a H\"older continuous density $\varrho$ of the solution $\mu$
with respect to Lebesgue measure.
Moreover, for every ball $U\subset\mathbb{R}^d$ and
every closed interval $J\subset(0, T)$ we have $\varrho(\, \cdot \,, t)\in W^{1, p}(U)$ and
$$
\int_{J}\|\varrho(\, \cdot \,, t)\|_{W^{1, p}(U)}^p\,dt<\infty.
$$

Set $B^i=b^i-\partial_{x_j}a^{ij}$. Then we can rewrite equation (\ref{e1}) in the divergence form
\begin{equation}\label{e3.1}
\partial_t\varrho={\rm div}\bigl(A\nabla\varrho-B\varrho\bigr)+c\varrho,
\end{equation}
which is understood in the sense of the  integral identity
\begin{equation}\label{id3}
\int_0^T\int_{\mathbb{R}^d}\bigl[
-\varrho\partial_t\varphi+(A\nabla\varrho, \nabla\varphi)\bigr]\,dx\,dt=
\int_0^T\int_{\mathbb{R}^d}\bigl[(B, \nabla\varphi)\varrho+c\varrho\varphi\bigr]\,dx\,dt
\end{equation}
for every function $\varphi\in C^{\infty}_0(\mathbb{R}^d\times(0, T))$.

Recall the following embedding theorem (see \cite[Lemma 3.1]{BRSH1} or \cite{ArS}).

\begin{lemma}\label{lem1}
Let $J$ be a closed interval in $(0, T)$ and let $u(\,\cdot\,, t)\in W^{1, 2}(\mathbb{R}^d)$
be such that $x\mapsto u(x, t)$ has compact support for almost all $t\in J$.
Then there exists a constant $C>0$ depending only on $d$ such that
$$
\|u\|_{L^{2(d+2)/d}(\mathbb{R}^d\times J)}\le C\bigl(\sup_{t\in J}\|u(\,\cdot\,, t)\|_{L^2(\mathbb{R}^d)}
+\|\nabla u\|_{L^2(\mathbb{R}^d\times J)}\bigr).
$$
\end{lemma}

Note that now  we  do not assume that $c$ is a nonpositive function.

Let $c^{+}(x, t)=\max\{c(x, t), 0\}$.

The following lemma is the key step of our proof.

\begin{lemma}\label{lem2}
Let $m\ge 1$. Let $U\subset\mathbb{R}^d$ be a ball and let $[s_1, s_2]\subset (0, T)$.
Assume that $\psi\in C^{\infty}_0(\mathbb{R}^d\times(0, T))$ is such that
 the support of $\psi$ is contained in $U\times(0, T)$
and $\psi(x, s_1)=0$ for every $x$. Then
there exists a constant $C(d)$ depending only on $d$ such that
\begin{multline}\label{ineq3}
\Bigl(\int_{s_1}^{s_2}\int_{U}|\varrho^m\psi|^{2(d+2)/d}\,dx\,dt\Bigr)^{d/(d+2)}\le
\\
\le 32C(d)m^2(1+\lambda^{-1})
\int_{s_1}^{s_2}\int_{U}\bigl[|\psi||\psi_t|+\|A\||\nabla\psi|^2
+|\sqrt{A^{-1}}B|^2\psi^2+c^{+}\psi^2\bigr]\varrho^{2m}\,dx\,dt,
\end{multline}
where $\|A(x, t)\|=\min_{|\xi|=1}(A(x, t)\xi, \xi)$ and $\lambda=\lambda(U, [s_1, s_2])$
is defined as above.
\end{lemma}
\begin{proof}
Let $f$ be a smooth function on $[0, +\infty)$ such that  $f\ge 0$, $f'\ge 0$, $f''\ge 0$.
Substituting the function $\varphi=f'(\varrho)\psi^2$ in equality (\ref{id3}),
 for any $t\in[s_1, s_2]$ we obtain
\begin{multline*}
\int_{\mathbb{R}^d}f(\varrho(x, t))\psi^2(x)\,dx-\int_{\mathbb{R}^d}f(\varrho(x, s_1))\psi^2(x)\,dx+
\\
+\frac{1}{3}\int_{s_1}^t\int_{\mathbb{R}^d}|\sqrt{A}\nabla\varrho|^2f''(\varrho)\psi^2\,dx\,d\tau\le
\\
\int_{s_1}^t\int_{\mathbb{R}^d}2|\psi||\psi_t|f(\varrho)+3|\sqrt{A}\nabla\psi|^2\frac{f'(\varrho)^2}{f''(\varrho)}
+3|\sqrt{A^{-1}}B|^2\varrho^2f''(\varrho)\psi^2+
\\
+2|(B, \nabla\psi)|\psi\varrho f'(\varrho)+c^{+}\varrho f'(\varrho)\psi^2\,dx\,d\tau.
\end{multline*}
To this end, it is enough to note that
$$
2(A\nabla\varrho, \nabla\psi)\psi f'(\varrho)\le 3^{-1}|\sqrt{A}\nabla\varrho|^2f''(\varrho)\psi^2+
3|\sqrt{A}\nabla\psi|^2\frac{f'(\varrho)^2}{f''(\varrho)},
$$
$$
(B, \nabla\varrho)\varrho f''(\varrho)\psi^2\le 3^{-1}|\sqrt{A}\nabla\varrho|^2f''(\varrho)\psi^2+
3|\sqrt{A^{-1}}B|^2\varrho^2f''(\varrho)\psi^2.
$$
Set $f(\varrho)=\varrho^{2m}$. Recall that $\psi(x, s_1)=0$. We have
\begin{multline*}
\sup_{t\in [s_1, s_2]}\int_{\mathbb{R}^d}\varrho^{2m}(x, t)\psi^2(x)\,dx+
\frac{4m-2}{3m}\int_{s_1}^{s_2}\int_{\mathbb{R}^d}|\sqrt{A}\nabla(\varrho^m\psi)|^2\,dx\,d\tau\le
\\
\le32m^2\int_{s_1}^{s_2}\int_{\mathbb{R}^d}\bigl[|\psi||\psi_t|+|\sqrt{A}\nabla\psi|^2
+|\sqrt{A^{-1}}B|^2\psi^2+c^{+}\psi^2\bigr]\varrho\,dx\,d\tau.
\end{multline*}
Now our assertion follows from Lemma \ref{lem1}.
\end{proof}

\begin{theorem}\label{th3-1}{\rm ($L^p$-estimates)}
Let $p\ge 2(d+2)/d$. Let $U$ and $U'$ be balls in $\mathbb{R}^d$ with $\overline{U'}\subset U$.
Let also $[s_1, s_2]\subset (0, T)$.
Then, for every $s\in (s_1, s_2)$, there exists a number
$C>0$ depending on $U$, $U'$, $s$, $s_1$, $d$ and $p$
such that
$$
\|\varrho\|_{L^p(U'\times[s, s_2])}\le C(1+\lambda^{-1})^{\gamma}
\int_{s_1}^{s_2}\int_{U}\bigl[1+\|A\|^{\gamma}
+|c^{+}|^{\gamma}+|\sqrt{A^{-1}}B|^{2\gamma}\bigr]\varrho\,dx\,dt,
$$
where $\gamma=(d+2)/2p'$, $p'=p/(p-1)$
and $\lambda=\lambda(U, [s_1, s_2])$, $\|A\|$ are defined as above.
\end{theorem}
\begin{proof}
Set $m=dp/2(d+2)$ and
$$
\alpha=1+\frac{4m}{(2m-1)d}, \quad \alpha'=1+\frac{(2m-1)d}{4m}, \quad \delta=\frac{4}{d(2m-1)+4m}.
$$
Note that $m\ge 1$.
Let us fix a function $\psi=\zeta(x)\eta(t)$, where $\zeta\in C^{\infty}_0(U)$,
$\zeta(x)=1$ if $x\in U'$, $0\le \psi\le 1$, $\eta\in C^{\infty}_0((s_1, T))$,
$\eta(t)=1$ if $t\in [s, s_2]$, $0\le \eta\le 1$ and
$$
|\partial_t\eta(t)|\le K\eta^{1-\delta}(t), \quad
|\nabla\zeta(x)|\le K\zeta^{1-\delta}(x)
$$
for some number $K>0$ and every $(x, t)\in U\times[s_1, s_2]$.
Note that $K$ depends only on $U$, $U'$, $s$ and $s_1$.
Applying Lemma \ref{lem2} we obtain
\begin{multline*}
\Bigl(\int_{s_1}^{s_2}\int_{U}|\varrho^m\psi|^{2(d+2)/d}\,dx\,dt\Bigr)^{d/(d+2)}\le
\\
\le32C(d)m^2(1+\lambda^{-1})\int_s^{s_2}\int_{U}\bigl[|\psi||\psi_t|+\|A\||\nabla\psi|^2+
|\sqrt{A^{-1}}B|^2\psi^2+c^{+}\psi^2\bigr]\varrho^{2m}\,dx\,dt.
\end{multline*}
Using H\"older's inequality with exponents
$\alpha$ and $\alpha'$, we estimate the integral in the right side of the last inequality
by the following expression:
$$
K^2\Bigl(\int_{s_1}^{s_2}\int_{U}\bigl(1+\|A\|+
|\sqrt{A^{-1}}B|^2+c^{+}\bigr)^{\alpha'}\varrho^{2m}\,dx\,dt\Bigr)^{1/\alpha'}
\Bigl(\int_{s_1}^{s_2}\int_{U}|\varrho^m\psi|^{2(d+2)/d}\,dx\,dt\Bigr)^{1/\alpha}.
$$
Applying the inequality $xy\le \varepsilon x^{\alpha}+C(\alpha, \varepsilon)y^{\alpha'}$
with sufficiently small $\varepsilon>0$, we obtain our assertion.
\end{proof}

\begin{theorem}\label{th3-2}{\rm ($L^{\infty}$-estimates)}
Let $\gamma>(d+2)/2$.
Let $U$ and $U'$ be balls in $\mathbb{R}^d$ with $\overline{U'}\subset U$.
Let also $[s_1, s_2]\subset (0, T)$.
Then, for every $s\in (s_1, s_2)$, there exists a number
$C>0$ depending on $U$, $U'$, $s$, $s_1$, $d$ and $\gamma$ such that
$$
\|\varrho\|_{L^{\infty}(U'\times[s, s_2])}\le C(1+\lambda^{-1})^{\gamma}
\int_{s_1}^{s_2}\int_{U}\bigl[1+\|A\|^{\gamma}+|c^{+}|^{\gamma}+|\sqrt{A^{-1}}B|^{2\gamma}\bigr]\varrho\,dx\,dt,
$$
where $\lambda=\lambda(U, [s_1, s_2])$, $\|A\|$ are defined as above.
\end{theorem}
\begin{proof}
If $\varrho\equiv 0$ on $U\times[s_1, s_2]$, then the assertion is trivial.
Let us consider the case where $\varrho\not\equiv 0$.
Multiplying the solution $\varrho$ by the number
$$
(1+\lambda^{-1})^{-\gamma}\Bigl(
\int_{s_1}^{s_2}\int_{U}\bigl[1+\|A\|^{\gamma}+|c^{+}|^{\gamma}+|\sqrt{A^{-1}}B|^{2\gamma}\bigr]\varrho\,dx\,dt\Bigr)^{-1},
$$
we  can assume that
$$
(1+\lambda^{-1})^{\gamma}
\int_{s_1}^{s_2}\int_{U}\bigl[1+\|A\|^{\gamma}+|c^{+}|^{\gamma}+|\sqrt{A^{-1}}B|^{2\gamma}\bigr]\varrho\,dx\,dt=1.
$$
In this case in order to prove the theorem it is enough to find
a number $C$ depending only on $U$, $U'$, $s$, $s_1$, $s_2$, $d$ and $\gamma$
such that
$$
\|\varrho\|_{L^{\infty}(U'\times[s, s_2])}\le C.
$$
Let $U=U(x_0, R)$, $U'=U(x_0, R')$ and $R'<R$.
Set $R_n=R'+(R-R')2^{-n}$, $s_n=s-(s-s_1)2^{-n}$ and $U_n=U(x_0, R_n)$.
Let us consider the following system of increasing domains:
$$
Q_n=U_n\times[s_n, s_2], \quad Q_0=U\times[s_1, s_2].
$$
For each $n$ we fix a function $\psi_n\in C^{\infty}_0(\mathbb{R}^d\times(0, T))$
in the same way as in the proof of Theorem \ref{th3-1}, that is,
$\psi(x, t)=1$ if $(x, t)\in Q_{n+1}$, $0\le \psi\le 1$,
the support of $\psi$ is contained in $U_n\times (s_n, T)$
and $|\partial_t\psi_n(x, t)|+|\nabla\psi_n(x, t)|\le K^n$
for all $(x, t)\in \mathbb{R}^d$ and some number $K>1$
depending only on the numbers $s$, $s_1$, $R$, $R'$.

Applying Lemma \ref{lem2} and  H\"older's inequality
with exponents $\gamma$ and $\gamma'$, we obtain
$$
\Bigl(\int_{Q_n}|\varrho^m\psi_n|^{2(d+2)/d}\,dx\,dt\Bigr)^{d/(d+2)}\le
32m^2C(d, s)K^{2n}\Bigl(\int_{Q_n}\varrho^{(2m-1)\gamma'+1}\,dx\,dt\Bigr)^{1/\gamma'}.
$$
Set
$$
p_{n+1}=\beta p_n+(\gamma'-1)\gamma'^{-1}, \quad p_1=\gamma'+1, \quad
\beta=(d+2)d^{-1}\gamma'^{-1}.
$$
Note that $\beta^{n-1}p_1\le p_n\le \beta^{n-1}(p_1+1)$.
Taking $m=p_{n+1}d/(2d+4)$, we obtain
$$
\|\varrho\|_{L^{p_{n+1}}(Q_{n+1})}\le C^{n\beta^{-n}}\|\varrho\|_{L^{p_n}(Q_n)}^{p_n/(p_n+\gamma'-1)},
$$
where the number $C$ depends only on $K$, $d$, and $\gamma$.
Finally, note that $\sum_n n\beta^{-n}<\infty$ and according to Theorem \ref{th3-1}
the norm $\|\varrho\|_{L^{p_1}(Q_1)}$ is estimated by a number
depending only on the numbers~$p_1$, $d$, $s$, $s_1$, $U$, and $U_1$.
\end{proof}

\begin{remark}\rm
 (i) Note that the constant $C$ in Theorem \ref{th3-1} and Theorem \ref{th3-1} does not depend on~$s_2$.

 (ii)  If $c\le 0$, then all the  inequalities above will be true without
the coefficient $c$ in the right-hand side.
\end{remark}

\begin{corollary}
Let $\gamma>(d+2)/2$, $\kappa>0$ and $t_0\in (0, T)$. Then there exists a number $C>0$ depending only on
$\kappa$, $t_0$, $d$ and $\gamma$ such that for all $(x, t)\in\mathbb{R}^d\times(t_0, T)$
$$
\varrho(x, t)\le
C(1+\lambda^{-1}(x, t))^{\gamma}\int_{t_0/2}^t\int_{U(x, \kappa)}
(1+\|A\|^{\gamma}+|c^{+}|^{\gamma}+|\sqrt{A^{-1}}B|^{2\gamma})\varrho\,dy\,d\tau,
$$
where
$$
\lambda(x, t)=\inf\bigl\{(A(y, \tau)\xi, \xi)\colon
 \, |\xi|=1, \quad (y, \tau)\in U(x, \kappa)\times[t_0/2, t]\bigr\}.
$$
In particular, if $\mu_t(dx)=\varrho(x, t)\,dx$ is a subprobability measure for almost all $t\in(0, T)$,
the functions $\|A\|^{\gamma}$, $|c^{+}|^{\gamma}$, $|B|^{2\gamma}$
are in $L^1(\mathbb{R}^d\times(t_0/2, T), \mu)$ and
the function $\|A\|^{-1}$ is uniformly bounded, then $\varrho\in L^{\infty}(\mathbb{R}^d\times(t_0, T))$.
\end{corollary}
\begin{proof}
Let us shift the point $x$ to $0$ and apply Theorem \ref{th3-2}
with the balls $U=U(x, \kappa)$ and $U'=U(x, \kappa/2)$ and points
$s_1=t_0/2$, $s=t_0$, $s_2=t$.
\end{proof}

\begin{corollary}
Let $\gamma>(d+2)/2$ and $\Theta\in (0, 1)$. Then there exists a number $C>0$ depending only on
$\gamma$, $d$ and $\Theta$ such that for all $(x, t)\in \mathbb{R}^d\times(0, T)$
$$
\varrho(x, t)\le
C(1+\lambda^{-1}(x, t))^{\gamma}t^{-(d+2)/2}\int_{\Theta t}^{t}\int_{U(x, \sqrt{t})}
(1+\|A\|^{\gamma}+t^{2\gamma}|c^{+}|^{\gamma}+t^{2\gamma}|\sqrt{A^{-1}}B|^{2\gamma})\varrho\,dy\,d\tau,
$$
where
$$
\lambda(x, t)=\inf\bigl\{(A(y, \tau)\xi, \xi)\colon \, |\xi|=1, \quad
(y, \tau)\in U(x, \sqrt{t})\times[\Theta t, t]\bigr\}.
$$
In particular, if $\mu_t(dx)=\varrho(x, t)\,dx$ is a subprobability measure for almost all $t\in(0, T)$,
the functions $\|A\|^{\gamma}$, $|c^{+}|^{\gamma}$, $|B|^{2\gamma}$
are in $L^1(\mathbb{R}^d\times(0, T), \mu)$ and
the function $\|A\|^{-1}$ is uniformly bounded, then there exists a number $\widetilde{C}>0$ such that
$$
\varrho(x, t)\le \widetilde{C}t^{-d/2}
\quad
\hbox{for all $(x, t)\in\mathbb{R}^d\times(0, T)$.}
$$
\end{corollary}
\begin{proof}
In order to prove the estimate at a point $(x_0, t_0)$
it suffices to change variables
$x \mapsto (x-x_0)/\sqrt{t_0}$ and $t\mapsto t/t_0$ and
apply Theorem \ref{th3-2} with the balls
$U=U(0, 1)$ and $U'=U(0, 1/2)$ and points $s_1=\Theta$, $s=(1+\Theta)/2$, $s_2=1$.
\end{proof}

\begin{corollary}\label{cor3}
Let $\Phi\in C^{2, 1}(\mathbb{R}^d\times(0, T))$ and $\Phi>0$.
Set
$$
\widetilde{c}=c+
\bigl(\partial_t\Phi+{\rm div}(A\nabla\Phi)+B\nabla\Phi\bigr)\Phi^{-1},
\quad
\widetilde{B}=B+\Phi^{-1}A\nabla\Phi.
$$
Let $\gamma>(d+2)/2$ and $\Theta\in(0, 1)$.
Then there exists a number $C>0$ depending only on
$\gamma$, $d$ and $\Theta$ such that for all $(x, t)\in \mathbb{R}^d\times(0, T)$
\begin{multline*}
\varrho(x, t)\le C\Phi(x, t)^{-1}(1+\lambda^{-1}(x, t))^{\gamma}\times
\\
\times
t^{-(d+2)/2}\int_{\Theta t}^{t}\int_{U(x, \sqrt{t})}
(1+\|A\|^{\gamma}+t^{2\gamma}|\widetilde{c}^{+}|^{\gamma}
+t^{2\gamma}|\sqrt{A^{-1}}\widetilde{B}|^{2\gamma})\Phi\varrho\,dy\,d\tau,
\end{multline*}
where $\lambda$ is defined in the previous corollary.
In particular, if
$$
\sup_{t\in (0, T)}\int_{\mathbb{R}^d}\Phi(x, t)\varrho(x, t)\,dx<\infty,
$$
the functions $\|A\|^{\gamma}\Phi$, $|\widetilde{c}^{+}|^{\gamma}\Phi$,
$|\widetilde{B}|^{2\gamma}\Phi$
are in $L^1(\mathbb{R}^d\times(0, T), \mu)$ and
the function $\|A\|^{-1}$ is uniformly bounded, then there exists a number $\widetilde{C}>0$ such that
$$
\varrho(x, t)\le \widetilde{C}t^{-d/2}\Phi(x, t)^{-1}
\quad \hbox{for all $(x, t)\in\mathbb{R}^d\times(0, T)$.}
$$
\end{corollary}
\begin{proof}
It suffices to observe that the function $\Phi\varrho$ satisfies equation (\ref{e3.1}) with
the new coefficients $\widetilde{c}$ and $\widetilde{B}$.
\end{proof}

Let us now consider two typical examples.
We shall assume that $c\le 0$ and that $\mu_t(dx)=\varrho(x, t)\,dx$ is a subprobability
solution  of the Cauchy problem for equation~(\ref{e1})
with the initial condition $\nu$ such that $|c|\in L^1(\mu)$ and
$$
\mu_t(\mathbb{R}^d)\le \nu(\mathbb{R}^d)+\int_0^t\int_{\mathbb{R}^d}c(x, s)\,\mu_s(dx)\,ds.
$$
We obtain upper estimates of $\varrho$ in several different situations.

\begin{example}\label{ex55}\rm
Let $\alpha>0$, $r>2$ and $k>r$. Assume that $c\le 0$ and
\begin{multline*}
\alpha r|x|^{r-2}\hbox{\rm trace}\, A(x, t)+
\\
+\alpha r(r-2)|x|^{r-4}(A(x, t)x, x)+\alpha^2r^2|x|^{2r-4}(A(x, t)x, x)+
\\
+\alpha r|x|^{r-2}(b(x, t), x)+c(x, t)\le C-C|x|^k
\end{multline*}
for some $C>0$ and all $(x, t)\in\mathbb{R}^d\times(0, T)$. Suppose also that
for all $(x, t)\in\mathbb{R}^d\times(0, T)$ we have
$$
C_1\exp(-\kappa_1|x|^{r-\delta})\le \|A(x, t)\|\le C_2\exp(\kappa_2|x|^{r-\delta}),
$$
and
$$
|b^i(x, t)|+|\partial_{x_j}a^{i j}(x, t)|\le C_3\exp(\kappa_3|x|^{r-\delta})
$$
with some positive numbers $C_1$, $C_2$, $C_3$, $\kappa_1$, $\kappa_2$, $\kappa_3$ and $\delta\in (0, r)$.
Let $\alpha'\in (0, \alpha)$.
Then the density $\varrho$ satisfies the  inequality
$$
\varrho(x, t)\le C_4\exp(-\alpha'|x|^r)\exp(C_5 t^{-\frac{r}{k-r}})
$$
for all $(x, t)\in\mathbb{R}^d\times(0, T)$ and some positive numbers $C_4$ and $C_5$.
\end{example}
\begin{proof}
According to Example \ref{3exm03} we have
$$
\int_{\mathbb{R}^d}\exp\bigl(\alpha|x|^r\bigr)d\mu_t\le
\gamma_1\exp\bigl(\gamma_2t^{-\frac{r}{k-r}}\bigr)
$$
for almost every $t\in(0, T)$ and some numbers $\gamma_1$ and $\gamma_2$.
Set $\Phi(x)=\exp(\alpha'|x|^r)$.
Note that $\widetilde{c}^{+}\le \gamma_3$ and
$$
(1+\|A\|^{\gamma}
+t^{2\gamma}|\sqrt{A^{-1}}\widetilde{B}|^{2\gamma})\Phi\le \gamma_4\exp(\alpha|x|^r)
$$
for all $(x, t)\in\mathbb{R}^d\times(0, T)$ and some number $\gamma_3$.
Now the desired estimates follow from Corollary \ref{cor3}.
\end{proof}

\begin{example}\rm
Let $r>2$, $k>2$, $\gamma>d+2$, $\alpha>0$ and $\beta>r/(k-2)$. Assume that
\begin{multline*}
\alpha r\hbox{\rm trace}\, A(x, t)
+\alpha r(r-2)|x|^{-2}(A(x, t)x, x)+
\\
+\alpha r(b(x, t), x)+|x|^2c(x, t)+\alpha^2r^2|x|^{r-2}(A(x, t)x, x)\le C-C|x|^k,
\end{multline*}
where $C>0$. Suppose also that
for all $(x, t)\in\mathbb{R}^d\times(0, T)$ we have
$$
C_1(1+|x|^{\frac{m}{\gamma}})^{-1}\le \|A(x, t)\|\le C_2(1+|x|^{\frac{m}{\gamma}})
$$
and
$$
|b^i(x, t)|^{2\gamma}+|\partial_{x_j}a^{ij}(x, t)|^{2\gamma}\le C_3(1+|x|^m)
$$
with some positive numbers $C_1$, $C_2$, $C_3$ and $m\ge \gamma\max\{r-1, r\beta^{-1}\}$.
Let $\alpha'\in (0, \alpha)$. Then the density $\varrho$ satisfies the  inequality
$$
\varrho(x, t)\le C_4t^{-\frac{8m\beta+rd-4\gamma r}{2r}}\exp(-\alpha' t^{\beta}|x|^r)
$$
for all $(x, t)\in\mathbb{R}^d\times(0, T)$ and some positive numbers $C_4$ and $C_5$.
\end{example}
\begin{proof}
According to Example \ref{3exm04} we have
$$
\int_{\mathbb{R}^d}\exp(\alpha t^{\beta}|x|^r)\,d\mu_t\le \gamma_1
$$
for all $t\in(0, T)$ and some number $\gamma_1$.
Note that for every $p\ge 1$ and $\varepsilon>0$ one has
$$
|x|^p\le \gamma_2t^{-\frac{\beta p}{r}}\exp(\varepsilon t^{\beta}|x|^r),
$$
so we can apply Corollary \ref{cor3} with $\Phi(x, t)=\exp(\alpha t^{\beta}|x|^r)$.
\end{proof}

{\bf Acknowledgements}

The author expresses his deep gratitude to V.I. Bogachev for fruitful discussions
and remarks. The author thanks the organizers for a very nice conference in Gaeta.

This work was supported by the RFBR projects
11-01-00348-a, 12-01-33009,
11-01-12018-ofi-m-2011,
and the program SFB 701 at the University of Bielefeld.

\end{document}